\documentclass[11pt]{amsart}

\usepackage{epigamath}

\usepackage[english]{babel}

\numberwithin{equation}{section}

\usepackage[shortlabels]{enumitem}

\theoremstyle{plain}
\newtheorem{theorem}{Theorem}[section]
\newtheorem{proposition}[theorem]{Proposition}
\newtheorem{corollary}[theorem]{Corollary}

\theoremstyle{remark}
\newtheorem{remark}[theorem]{Remark}

\newcommand {\Morifan}{\operatorname{MF}}
\newcommand{\PMod}{\operatorname{PMod}}
\newcommand{\sY}{{\mathcal Y}}
\newcommand{\Z}{{\mathbb Z}}

\newcommand{\T}{\mathscr{T}\!}

\setlist[enumerate,1]{label={\rm(\roman*)}, ref={\rm\roman*}}
\setlist[enumerate,2]{label={\rm(\arabic*)}, ref={\rm\arabic*}}

\EpigaVolumeYear{10}{2026} \EpigaArticleNr{17} \ReceivedOn{June 13, 2025}
\InFinalFormOn{November 18, 2025}
\AcceptedOn{December 15, 2025}

\title{Corrigendum to \\
  ``The Mori fan of the Dolgachev--Nikulin--Voisin family\\
  in genus 2'' by K.~Hulek and C.~Liese}
\titlemark{Corrigendum}

\author{Mathieu Dutour Sikiri\'c}
\address{MSM-programming, Karlova\v{c}ka Cesta 28B, 10452 Klin\v{c}a Selo, Croatia}
\email{mathieu.dutour@gmail.com}
\author{Klaus Hulek}
\address{Institut f\"ur Algebraische Geometrie, Leibniz Universit\"at Hannover, Welfengarten 1, 30167 Hannover, Germany}
\email{hulek@math.uni-hannover.de}
\author{Christian Lehn}
\address{Fakult\"at f\"ur Mathematik, Ruhr-Universit\"at Bochum,  Universit\"ats\-stra\ss e~150, Postfach IB 45, 44801 Bochum, Germany}
\email{christian.lehn@rub.de}

\authormark{M.~Dutour Sikiri\'c, K.~Hulek, and C.~Lehn}

\AbstractInEnglish{In this note, we correct some of the results of the paper ``The Mori fan of the Dolgachev--Nikulin--Voisin family in genus $2$'' by K.~Hulek and C.~Liese concerning the number of maximal cones in the Mori fan of the Dolgachev--Nikulin--Voisin fan in genus $2$.  The errors in the original paper concern the correct enumeration of cones. The method and the main theoretical results are not affected.}


\acknowledgement{Klaus Hulek has been partially supported by DFG grant Hu 337/7-2, and Christian Lehn has been partially supported by DFG grants Le 3093/2-2 and Le 3093/3-1.}

\begin{document}



\maketitle

\begin{prelims}

\DisplayAbstractInEnglish

\bigskip

\DisplayKeyWords

\medskip

\DisplayMSCclass

\end{prelims}


\newpage

\setcounter{tocdepth}{1}

\tableofcontents


\section{Introduction}\label{section intro}

The principal purpose of \cite{HL22} was to study the Mori fan $\Morifan(\sY/S)$ of the Dolgachev--Nikulin--Voisin family of degree $2$. One of the main results of this paper is an enumeration of all maximal cones of this fan. When using the results of \cite{HL22} for further studies, we noticed that \cite{HL22} contains some enumeration errors. The mistakes made are in tracing the different cases and are of a bookkeeping nature. The fundamental approach and the main theoretical results of \cite{HL22} are not affected. We found the mistakes by using a computer program which can be used to list all possible models; see Remark~\ref{remark computer program}. Comparing this list with the results in \cite{HL22}, we found that we had overcounted the number of models in some cases.  The required corrections can be checked independently of the computer program by hand, purely using the methods developed in \cite{HL22}. In this corrigendum, we provide the correct numbers and point out where the counting errors were made.

\section{Model \texorpdfstring{$\boldsymbol{\mathscr{T}}$}{T}}

We first correct the enumeration of the models of type $\T$. The number of $131$ models given in \cite[Theorem 7.7]{HL22} must be corrected as follows. 

\begin{theorem}\label{modelsT}  There are $129$ surfaces in $\PMod_2(\mathscr{T})$.
\end{theorem}

\begin{proof}
The error is in the enumeration for the case $n_1=2$. In the first case, the number of $(-1)$-curves which can be flopped is given by $r_1 \in \{0,\ldots,8\}$, and in the second case, we have $r_2 \in \{0,\ldots,7\}$. The cases $r_1=9$ and $r_2=8$ do not exist, as there is no $(-1)$-curve left which can be contracted. Hence, this gives us $19$ rather than $21$ cases. Altogether, we find $30+19+34 = 83$ models for $n_1 \geq -1$ and $n_2 \leq -1$. The other cases were enumerated correctly. In particular, we get $1$ model for $n_1,n_2 \geq 0$, namely $(n_1,n_2)=(0,0)$, and $36$ models for $n_1,n_2 \leq -2$.  This gives a total of $83 + 1 + 45=129$ models.
\end{proof}

\section{Model \texorpdfstring{$\boldsymbol{\mathscr{P}}$}{P}}

Here the enumeration of the models is divided into several parts. The first correction concerns \cite[Theorem 7.8(ii)]{HL22}, where we listed one model too many. The correct version is as follows. 

\begin{theorem}\label{modelsP}
There are $103$ surfaces $Y_c=Y_1\cup Y_2 \cup Y_3 $ in $\PMod_2(\mathscr{P})$ such that there is a component $Y_i$ with very degenerate curve structure. Explicitly, these are given as follows:
\begin{enumerate}
\item\label{modelsP-1} Surfaces $Y_c$ such that $\Gamma_{Y_1}$ is very degenerate, $\Gamma_{Y_2}$ is non-degenerate and $\Gamma_{Y_3}$ is tamely degenerate. There are $71$ such models.
\item\label{modelsP-2} Surfaces $Y_c$ such that $\Gamma_{Y_1}$ is very degenerate, $\Gamma_{Y_2}$ is non-degenerate and $\Gamma_{Y_3}$ is very degenerate. There are $7$ such models.
\item\label{modelsP-3} Surfaces $Y_c$ such that $\Gamma_{Y_1}$ is very degenerate, $\Gamma_{Y_2}$ is non-degenerate and $\Gamma_{Y_3}$ is non-degenerate. There are $25$ such models.
\end{enumerate}
\end{theorem}

\begin{proof}
The error lies in the first subcase of~(\ref{modelsP-2}). We have $D_{13}^2\in [-3,-1]$, giving only 3 cases (instead of $D_{13}^2\in [-3,0]$, giving 4 cases). More precisely, we have $v_{D_{12}} \cdot D_{12}=v_{D_{32}}\cdot D_{32}=2$ (and not $v_{D_{31}}\cdot D_{31}=2$ as stated). Projectivity is guaranteed by \cite[Proposition 5.3(i)]{HL22} (which applies with a shift of indices).  The possible values for $(D_{13}^2,D_{31}^2)$ are $(-1+k,-1-k)$, where $k \in \{-2,-1,0,1,2\}$. Since $k= \pm 1$ and $k=\pm 2$ give isomorphic surfaces, we have 3 models here.
\end{proof}

The next step is to correct \cite[Theorem 7.12]{HL22}. In (i) and (iii) we counted 2, respectively 6, surfaces double, not taking into account that they are isomorphic.  Also, the parameter sets in (ii) and (iii) must be slightly modified, but this does not change the number of cones.

\begin{theorem}\label{modelsPreg}
There are $25+103+219=347$ surfaces $Y=Y_1\cup Y_2\cup Y_3$ in $\PMod_2(\mathscr{P})$ such that the curve structures $\Gamma_{Y_i}$ are all regular. Explicitly, these are, up to equivalence, given as follows:
\begin{enumerate}
\item\label{modelsPreg-1} Surfaces with all $\Gamma_{Y_i}$ non-degenerate. These are given by the triples
\begin{align*}
&(0,1,-1),(0,1,2),(0,1,-2),(0,2,1),(0,2,-2),(0,-1,2),\\ &(0,-1,1),(0,-2,2),
(1,2,-1),(1,2,-2),(1,-1,2),(1,-2,2),\end{align*}  surfaces $(x,y,y)$ with $x\in \{1,2\}$ and $y\in  \{-2,-1,0,1,2\}\backslash \{x\}$, the surfaces $(0,1,1)$ and $(0,2,2)$ and surfaces $(x,x,x)$ with $x\in\{0,1,2\}$. These are $25$ surfaces.
\item\label{modelsPreg-2} Surfaces with one $\Gamma_{Y_i}$ degenerate: triples $(3,y,-3)$ with $0 \leq  y \leq 2$ and triples $(x,y,z)$ with $x,y \in \{-2,\ldots,2\}$, $z\in \{x-6,\ldots,-3\}$. These are $103$ surfaces.
\item\label{modelsPreg-3} Surfaces with two $\Gamma_{Y_i}$ degenerate. These are given by the sets  
\begin{align*}
&K = \{(x,-3,3) \mid 3 \leq x \leq 9\}\\
&M(0)=\{ (x,0,z) \mid -3\geq x \geq -6, 6\geq z\geq 3, |z| \leq |x| \}\\
&M(-1)=\{ (x,-1,z) \mid -3\geq x \geq -7, 5\geq z\geq 3\}\\ 
&M(-2)=\{ (x,-2,z) \mid -3\geq x \geq -8, 4\geq z\geq 3\}\\
&N(-2)=\{ (x,y,-2) \mid -3\geq y\geq -8 , -3\geq x\geq y-6\}\\
&N(-1)=\{ (x,y,-1) \mid -3\geq y\geq -7, -3\geq x\geq y-6\}\\
&N(0)=\{ (x,y,0) \mid -3\geq y\geq -6, -3\geq x\geq y-6\}\\
&N(1)=\{ (x,y,1) \mid -3\geq y\geq -5, -3\geq x\geq y-6\}\\
&N(2)=\{ (x,y,2) \mid -3\geq y\geq -4, -3\geq x\geq y-6\}.
\end{align*}
Adding up, these are  $7+10+15+12+57+45+34+24+15= 219$ surfaces.
\end{enumerate}
\end{theorem}

\begin{proof}
In (\ref{modelsPreg-1}) the enumeration given in \cite{HL22} included the triples $(x,x,x)$ with $x\in\{-2,-1,0,1,2\}$. However, the triples $(-1,-1,-1)$ and $(1,1,1)$, and also $(-2,-2,-2)$ and $(2,2,2)$, are equivalent since they differ by an involution in the sense of \cite[Definition 7.9]{HL22} and hence have isomorphic associated surfaces. Thus, we obtain 2 surfaces fewer.

In (\ref{modelsPreg-2}) we also note that $z\in \{y-6,\ldots,-3\}$ must be replaced by $z\in \{x-6,\ldots,-3\}$. This does not affect the enumeration.

In (\ref{modelsPreg-3}) we first note that $K=\{(x,-3,-3)\mid 3\leq x\leq 9\}$ must be replaced by $K=\{(x,-3,3)\mid 3\leq x\leq 9\}$. This does not affect the enumeration.  The error in \cite{HL22} was to set
$$
M(0)=\{ (x,0,z) \mid -3\geq x \geq -6, 6\geq z\geq 3 \}
$$
However, not all triples in this definition of $M(0)$ are inequivalent. The triple $(-6, 0, 3)$ is equivalent to $(0, 3,-6)$ by a shift in the sense of \cite[Definition 7.9]{HL22}, which in turn is equivalent to $(-3,0,6) \in M(0)$ by the involution of \cite[Definition 7.9]{HL22}. The same discussion applies to the triples $(-6, 0, 4)$, $(-6, 0, 5)$, $(-5, 0, 3)$, $(-5, 0, 4)$, $(-4, 0, 3)$. Hence $M(0)$ only contributes $10$ and not $16$ non-isomorphic surfaces, resulting in an overcount of 6 surfaces in \cite{HL22}.
\end{proof}

Combining Theorems~\ref{modelsP} and~\ref{modelsPreg}, we obtain the following. 

\begin{theorem}
\label{the:allmodelsP}
There are $103+347=450$ surfaces in $\PMod_2(\mathscr{P})$.
\end{theorem}

\begin{remark}\label{remark computer program}
While both the enumeration of models and the verification of the above claims can be performed by hand, we used a computer program for this task. It is based on the fact that the curve structures on the components of the central fiber uniquely determine the model by \cite[Sections 4.2 and 7.1]{HL22}. So in the program, a model can be represented by a labelled graph, and type I flops are implemented as operations on the graph. In a first step, we compute all models corresponding to cones in the Mori fan by performing type~I flops starting from a model in $(-1)$-form. To do this, we begin with a list containing a model in $(-1)$-form and successively add new models to the list by systematically performing all type I flops from models in the list, checking whether the resulting models are already in the list (graph isomorphism test), and appending the model if it is new. The search is complete if the list is closed under type I flops. In a second step, we test for degeneracy and regularity. Note that these properties are by definition determined by the graph.

We are currently expanding the capabilities of the program so that it can also explicitly calculate the cones of the Mori fan. We expect to release a version as part of an upcoming work. Independently of this, we want to stress again that the verification of the number of cones can be done by hand using the methods established in \cite{HL22} and are thus independent of this program.
\end{remark}

\section{Enumeration of maximal cones}

The above corrections also have consequences for the enumeration of the maximal cones. For this we have to analyse the symmetric cones, \textit{i.e.}, those cones, or equivalently 
models, which have non-trivial symmetries.

\begin{proposition}
\label{por::symmetricmodelsP}
There are $13$ models of type $\mathscr{P}$ which are symmetric. These are as follows: 
\begin{enumerate}
\item The model given by $(0,0,0)$ in Theorem~\ref{modelsPreg}\,\eqref{modelsPreg-1}. This has $S_3$-symmetry and hence orbit length $1$.
\item The $2$ models given by $(x,x,x), x \in \{1,2\}$ in Theorem~\ref{modelsPreg}\,\eqref{modelsPreg-1}. This has cyclic $\Z/3\Z$-symmetry and hence orbit length $2$. 
\item
  There are $10$ models which have a $\Z/2 \Z$-symmetry and hence orbit length $3$. Concretely, these are the following models: 
\begin{enumerate}
\item The $4$ models belonging to the triples $(0,-1,1), (0,1,-1), (0,-2,2), (0,2,-2)$ in Theorem~\ref{modelsPreg}\,\eqref{modelsPreg-1}.
\item The model belonging to the triple $(3,0,-3)$ in Theorem~\ref{modelsPreg}\,\eqref{modelsPreg-2}.
\item The $4$ models belonging to the triples $(-x,0,x), x \in \{3,4,5,6 \}$ in $M(0)$ in Theorem~\ref{modelsPreg}\,\eqref{modelsPreg-3}.
\item The model with $D_{12}^2=D_{23}^2=4$ and $D_{13}^2=-1$ in Theorem~\ref{modelsP}\,\eqref{modelsP-2}. 
\end{enumerate}
\end{enumerate}
\end{proposition}

\begin{proof}
It is clear that the surfaces listed have the symmetries as claimed. One can use the proof of \cite[Theorem 7.14]{HL22}, or a computer calculation, to show that these are the only symmetric models.
\end{proof}

\begin{remark}
Comparing this to \cite[Theorem 7.14]{HL22}, one sees that the orbits of length $2$, namely the triples $(x,x,x), x \in \{1,2\}$, were forgotten.
\end{remark}

\begin{corollary}
\label{cor:maximalconesPP}
There are $1 \times 1 + 2 \times 2 + 10 \times 3 + 437 \times 6= 2657$ maximal cones of type $\mathscr{P}$ in the Mori fan.
\end{corollary}

The number of symmetric cones of type $\T$ is stated correctly in \cite[Theorem~7.14]{HL22}. This implies the following. 

\begin{theorem}\label{maxconesT}  There are $741$ maximal cones of type $\T$ in $\Morifan(\mathscr{Y}/S)$.
\end{theorem}

\begin{proof}
There are now 129 models of type $\T$ of which 11 are symmetric (and have orbit length 3). This give $118 \times 6 + 11 \times 3= 747$ maximal cones.
\end{proof}

Combining these results finally gives us the following. 

\begin{theorem}\label{finalsum}
Let $\sY \to S$ be a model of the Dolgachev--Nikulin--Voisin family of degree $2$. Then the Mori fan $\Morifan(\sY/S)$ has $3398$ maximal cones. Of these $741$ are of type $\T$ and $2657$ are of type $\mathscr{P}$.
\end{theorem}


\end{document}